\documentclass[12pt,english]{article}
\usepackage[T1]{fontenc}
\usepackage[latin9]{inputenc}
\usepackage[paperwidth=21cm,paperheight=29.49cm]{geometry}
\usepackage{mathrsfs}
\usepackage{amsmath}
\usepackage{amsthm}
\usepackage{amssymb}
\usepackage[all]{xy}

\makeatletter
\numberwithin{equation}{section}
\numberwithin{figure}{section}
\theoremstyle{plain}
\newtheorem{thm}{\protect\theoremname}[section]
\theoremstyle{remark}
\newtheorem{rem}[thm]{\protect\remarkname}
\theoremstyle{plain}
\newtheorem{lem}[thm]{\protect\lemmaname}
\theoremstyle{plain}
\newtheorem{prop}[thm]{\protect\propositionname}
\theoremstyle{definition}
\newtheorem{defn}[thm]{\protect\definitionname}
\theoremstyle{remark}
\newtheorem*{rem*}{\protect\remarkname}
\theoremstyle{plain}
\newtheorem{conjecture}[thm]{\protect\conjecturename}
\theoremstyle{plain}
\newtheorem*{lem*}{\protect\lemmaname}

\usepackage{amscd}
\usepackage{youngtab}
\date{}

\makeatother

\usepackage{babel}
\providecommand{\conjecturename}{Conjecture}
\providecommand{\definitionname}{Definition}
\providecommand{\lemmaname}{Lemma}
\providecommand{\propositionname}{Proposition}
\providecommand{\remarkname}{Remark}
\providecommand{\theoremname}{Theorem}

\begin{document}
\title{On the exterior power structure of the cohomology groups for the general
hypergeometric integral}
\author{Hironobu Kimura\\
 Department of Mathematics, Graduate School of Science and\\
 Technology, Kumamoto University}

\maketitle
\global\long\def\C{\mathbb{C}}%
 
\global\long\def\a{\alpha}%
 
\global\long\def\Lc{\mathcal{L}^{\vee}}%
 
\global\long\def\mat#1{\mathrm{Mat}_{#1}(\C)}%
\global\long\def\G{\Gamma}%
 
\global\long\def\A{\mathcal{A}}%
\global\long\def\cL{\mathcal{L}}%
 
\global\long\def\na{\nabla}%
\global\long\def\W{\Omega}%
 
\global\long\def\w{\omega}%
 
\global\long\def\bu{\bullet}%
 
\global\long\def\GL#1{\mathrm{GL}_{#1}(\C)}%
\global\long\def\l{\lambda}%
\global\long\def\de{\delta}%
 
\global\long\def\La{\Lambda}%
\global\long\def\Z{\mathbb{Z}}%
\global\long\def\cbatu{\mathbb{C}^{\times}}%
\global\long\def\tH{\tilde{H}}%
\global\long\def\Ps{\mathbb{P}}%
\global\long\def\gras{\mathrm{G}_{r+1}(\C^{N})}%
\global\long\def\matt#1{\mathrm{Mat}'_{#1}(\C)}%
\global\long\def\zrn{Z_{r+1}}%
\global\long\def\znn{Z_{2}}%
\global\long\def\cO{\mathcal{O}}%
\global\long\def\lto{\longrightarrow}%
\global\long\def\bdy{\mathbf{y}}%
\global\long\def\f{\varphi}%
\global\long\def\sgn#1{\mathrm{sgn}(#1)}%
\global\long\def\ha{\mathfrak{h}}%
 
\global\long\def\cohomr{H^{r}(\Omega^{\bullet}(X^{r}),\boxtimes^{r}\nabla_{\omega})}%
 
\global\long\def\la{\langle}%
 
\global\long\def\ra{\rangle}%
 
\global\long\def\s{\sigma}%
 
\global\long\def\est#1{e_{#1}^{*}}%
\global\long\def\fraks{\mathfrak{S}}%
 
\global\long\def\cS{\mathcal{S}}%
\global\long\def\pa{\partial}%
 
\global\long\def\fj{\mathfrak{j}}%
\global\long\def\clam{\mathcal{M}_{_{\l}}}%
\global\long\def\cD{\mathcal{D}}%
\global\long\def\gl#1{\mathfrak{gl}_{#1}(\C)}%
 
\global\long\def\cohom#1#2#3{H^{#1}(\Omega^{\bullet}(#2),\nabla_{#3})}%

\global\long\def\bdx{\mathbf{x}}%
 
\global\long\def\bde{\mathbf{e}}%
 
\global\long\def\cali{\mathcal{I}}%
 
\global\long\def\cals{\mathcal{S}}%
 
\global\long\def\bdb{\mathbf{b}}%
 
\global\long\def\diag{\mathrm{diag}}%
 
\global\long\def\cW{\mathcal{W}}%
 
\global\long\def\tr{\mathrm{Tr}}%

\begin{abstract}
In this article we study the exterior power structure of the algebraic
de Rham cohomology group associated with the Gelfand hypergeometric
function and its confluent family. The hypergeometric function $F(z)$
is a function on the Zariski open subset $Z_{r+1}\subset\mat{r+1,N}$,
called the generic stratum, defined by an $r$-dimesional integral
on $\Ps^{r}$. For $z\in Z_{r+1}$, the algebraic de Rham cohomology
group is associated to the integral. When $z$ belongs to the particular
subset of $Z_{r+1}$, called the Veronese image, we show that this
cohomology group can be expressed as the exterior power product of
the de Rham cohomology group associated with the hypergeometric function
defined by $1$-dimensional integral.
\end{abstract}

\section{Introduction }

In this article, we study the exterior power structure of the algebraic
de Rham cohomology group associated with Gelfand's hypergeometric
function (HGF) \cite{Ge86} and its confluent family \cite{GRS88,KHT92,KimK96}.
These HGFs, which are called the general HGF in this paper and are
defined on the Grassmannian manifold $\mathrm{G}_{r+1}(\C^{N})$ of
linear subspaces of dimesion $r+1$ in $\C^{N}$ by the Radon transform,
give a far-reaching generalization of classical HGFs of one and several
variables. The general HGF provides a unified approach to the various
aspects of classical HGFs, like the transformation formulae, contiguity
relations, etc. \cite{HTF,KimK96}.

Firstly, let us explain the Gelfand HGF and its de Rham theoretic
aspect. Consider the integral

\begin{equation}
F(z)=\int_{C}U(u,z)\,du_{1}\wedge\cdots\wedge du_{r}\label{eq:intro-1}
\end{equation}
with
\begin{equation}
U(u,z)=f_{1}^{\a_{1}}\cdots f_{N}^{\a_{N}},\label{eq:intro-2}
\end{equation}
where $f_{1},\dots,f_{N}$ are degree-one polynomial functions on
$\C^{r}$ with the variables $u=(u_{1},\dots,u_{r})$:

\[
f_{j}=f_{j}(u,z)=z_{0j}+z_{1j}u_{1}+\cdots+z_{rj}u_{r},\;j=1,\dots,N,
\]
with coefficients $z_{0j},z_{1j,},\dots,z_{rj}\in\C$, and $\a_{1},\dots,\a_{N}$
are complex parameters, and $C=C(z)$ is an $r$-dimensional chain
in the $u$-space. If $C$ is appropriately chosen, $F(z)$ gives
a multivalued holomorphic function of $z=(z_{ij})$ on some Zariski
open subset of $\mat{r+1,N}$ \cite{Ao-Kita-book}. As particular
cases, we obtain the famous Gauss HGF and the Lauricella HGF $F_{D}$
as follows. Lauricella's $F_{D}$ is 

\begin{align*}
F_{D}(a,\bdb,c;x) & =\sum_{{\bf m}=(m_{1},\dots,m_{n})\in\Z_{\geq0}^{n}}\frac{(a)_{|{\bf m}|}(b_{1})_{m_{1}}\cdots(b_{n})_{m_{n}}}{(c)_{|{\bf m}|}\prod_{k=1}^{n}m_{k}!}x_{1}^{m_{1}}\cdots x_{n}^{m_{n}}\\
 & =\frac{\G(c)}{\G(a)\G(c-a)}\int_{0}^{1}u^{a-1}(1-u)^{c-a-1}\prod_{i=1}^{n}(1-x_{i}u)^{-b_{i}}du,
\end{align*}
where $(a)_{m}=\G(a+m)/\G(a)$ is the Pochhammer symbol and $|{\bf m}|=m_{1}+\cdots+m_{n}$.
This is a particular case of (\ref{eq:intro-1}) with (\ref{eq:intro-2}),
where $(r,N)=(1,n+3)$, and $(\a_{1},\a_{2},\dots,\a_{n+3})=(\a_{1},a-1,c-a-1,-b_{1},\dots,-b_{n})$,
and 
\begin{align*}
(f_{1},\dots,f_{n+3}) & =(1,u,1-u,1-x_{1}u,\dots,1-x_{n}u)\\
 & =(1,u)\left(\begin{array}{cccccc}
1 & 0 & 1 & 1 & \cdots & 1\\
0 & 1 & -1 & -x_{1} & \cdots & -x_{n}
\end{array}\right).
\end{align*}
The Gauss HGF is the one-variable case of $F_{D}$.

Before the work of Gelfand \cite{Ge86}, Aomoto \cite{Ao75,Ao82}
studied the HGF $F(z)$ using the de Rham theory. In particular, the
Pfaffian system characterizing $F$ was derived by using the de Rham
cohomology group. The de Rham theory in this case is explained as
follows. Fix $z=(z_{ij})$, and write $U(u,z),f_{j}(u,z)$ as $U(u),f_{j}(u)$,
respectively. Let $H_{j}=\{u\in\C^{r}\mid f_{j}(u)=0\},\,j=1,\dots,N,$
be the hyperplanes in $\C^{r}$. Then $\A=\left\{ H_{1},\dots,H_{N}\right\} $
gives a hyperplane arrangement in $\C^{r}$. Let $\cL$ be the local
system on $X=\C^{r}\setminus\cup_{j=1}^{N}H_{j}$ defined by the multivalued
function $U(u)^{-1}$ whose monodromy around $H_{j}$ is $\exp(-2\pi\sqrt{-1}\a_{j})$,
and $\Lc$ be its dual local system. We have the cohomology group
$H^{r}(X,\cL)$ with coefficients in $\cL$ and the homology group
$H_{r}(X,\Lc)$ with coefficients in $\Lc$. Then the integral (\ref{eq:intro-1})
is seen as a realization of the pairing of a class of $H^{r}(X,\cL)$
and a class $[C(z)\otimes U]$ of $H_{r}(X,\Lc)$. The cohomology
groups $H^{*}(X,\cL)$ are represented, by virtue of the comparison
theorem due to Deligne and Grothendieck \cite{Deligne,Griffiths-Harris},
as the cohomology groups of the complex
\begin{equation}
(\W^{\bu}(*\A),\na)=\{0\to\W^{0}(*\A)\overset{\na}{\to}\W^{1}(*\A)\overset{\na}{\to}\cdots\overset{\na}{\to}\W^{r}(*\A)\to0\},\label{eq:intro-3}
\end{equation}
where $\W^{p}(*\A)$ denotes the set of rational differential $p$-forms
on $\C^{r}$ having at most poles on $\A$, and $\na$ denotes the
twisted differential
\[
\na:=d+(d\log U(u))\wedge=d+(\sum_{j}\a_{j}\,d\log f_{j})\wedge.
\]
 The cohomology groups for this complex are called the (twisted) algebraic
de Rham cohomology groups, see \cite{Ao-Kita-book}. The properties
of the function $F$ reflect the topology of the arrangement $\A$
and the nature of parameters $\a$. So the de Rham theoretic aspect
of the Gelfand HGF attracted the attention of many authors \cite{A-K-Terao,Esnaut-V-Sch,orlik-terao}. 

In this paper, we are concerned with the general HGF including that
of confluent type \cite{KHT92,KimK96}(see Section 2 for the definition),
which includes, in a simple case, the integral representations of
Kummer's confluent HGF, Bessel function, Hermite-Weber function and
Airy function. For example, Kummer's function is given by

\[
_{1}F_{1}(a,c;x)=\sum_{m=0}^{\infty}\frac{(a)_{m}}{(c)_{m}m!}x^{m}=\frac{\G(c)}{\G(a)\G(c-a)}\int_{0}^{1}u^{a-1}(1-u)^{c-a-1}e^{xu}du,
\]
where we see the exponential function $e^{xu}$ of $u$ in the integrand. 

As is explained in Section 2, the general HGF, defined by an $r$-dimensional
integral (\ref{eq:intro-1}) whose integrand $U(u,z)$ is determined
by the matrix $z$ which lives in a Zariski open subset $Z\subset\mat{r+1,N}$
called the generic stratum (Definition \ref{def:hyp-2} ). The integrand
$U(u,z)$ has the form
\begin{equation}
U(u,z)=f_{1}(u,z)^{\a_{1}}\cdots f_{\ell}(u,z)^{\a_{\ell}}e^{g(u,z)},\label{eq:intro-4}
\end{equation}
where $f_{1},\dots,f_{\ell}$ are degree-one polynomial functions
of $u=(u_{1},\dots,u_{r})$ and give the hyperplane arrangement $\A=\left\{ H_{1},\dots,H_{\ell}\right\} $
in $\C^{r}$, and $g$ is a rational function of $u$ having poles
at most in $\cup_{j=1}^{\ell}H_{j}$. The functions $f_{1},\dots,f_{\ell}$
and $g$ are defined by the data $z\in Z$. 

In this case, the integral (\ref{eq:intro-1}) with (\ref{eq:intro-4})
can also be understood in the framework of the de Rham theory. Take
any point $z\in Z$ and fix it. For the cohomology, the power function
part $f_{1}^{-\a_{1}}\cdots f_{\ell}^{-\a_{\ell}}$ of $U(u)^{-1}$
defines the local system $\cL$ on $X:=\C^{r}\setminus\cup_{j=1}^{\ell}H_{j}$,
and the exponential part $e^{-g(u)}$ of $U(u)^{-1}$defines the family
$\Phi$ of closed subsets of $X$, called the family of support, on
which $e^{-g(u)}$ is rapidly decreasing. Then the cohomology for
(\ref{eq:intro-4}) is $H_{\Phi}^{\bullet}(X,\cL)$ which is called
the cohomology of $X$ with coefficients in the local system $\cL$
and with the family of supports $\Phi$, see \cite{pham1} p40 and
\cite{sabbah} p111 for the definition. It is shown \cite{sabbah}
that this cohomology group is isomorphic to the algebraic de Rham
cohomology defined by the (twisted) algebraic de Rham complex (\ref{eq:intro-3})
with the twisted differential
\[
\na=d+(d\log U)\wedge=d+(\sum_{j}\a_{j}\,d\log f_{j}+dg)\wedge.
\]

We are interested in the explicit computation of this algebraic de
Rham cohomology group for the general hypergeometric integral (HGI),
and we want to know, for example, its purity, the rank of top cohomology
group and explicit form of its basis.

We show in Theorem \ref{thm:main} that the de Rham cohomology groups
for the $r$-dimensional HGI have the exterior power structure at
particular points of $Z$ called Veronese points (Definition \ref{def:vero-3}).
It means, roughly speaking, that the cohomology groups for the $r$-dimensional
HGI can be expressed as the exterior product of the cohomology groups
for the one-dimensional HGI. As a result, we see that the purity of
the cohomology holds, namely $H^{p}(\W^{\bu}(*\A),\na)=0$ for $p\neq r$,
and that $\dim_{\C}H^{r}(\W^{\bu}(*\A),\na)=\binom{N-2}{r}$ at any
Veronese point. Moreover, at such point, we can construct a basis
of $H^{r}(\W^{\bu}(*\A),\na)$ explicitly from any basis of the top
cohomology group for the one-dimensional HGI (Propositions \ref{prop:exp-basis-1},
\ref{prop:Basis-4}). Note that the $r$-forms thus constructed at
a Veronese point are linearly independent in $H^{r}(\W^{\bu}(*\A),\na)$
for any $z$ belonging to some Zariski open subset of $Z\subset\mat{r+1,N}$
since the condition for the linear dependency gives a Zariski closed
condition for $z$. We expect that these $r$-forms give a basis of
$H^{r}(\W^{\bu}(*\A),\na)$ at any point $z\in Z$ and that these
$r$-forms will play important roles in the theory of general HGF,
for example, in the explicit computation of the cohomological intersection
numbers and its application, see \cite{Irina-Ki-Na,kimura-taneda,Kita-Yoshida}.
We mension that the exterior power structure for the Gelfand HGF (\ref{eq:intro-1}),
namely non-confluent case of the general HGF, was discussed  by Iwasaki
and Kita in \cite{iwa-kita}, see also \cite{tera}. We follow the
idea of \cite{iwa-kita} for the proof of the main theorem.

This paper is organized as follows. In Section 2, we explain breafly
the definition of the general HGF. The main theorem (Theorem \ref{thm:main})
is stated in Section 3 after defining the (generalized) Veronese map.
The proof of the main theorem is given in Section 4. Using the main
theorem, we give the explicit choice of the basis of the de Rham cohomology
$H^{r}(\W^{\bu}(*\A),\na)$ for a Veronese point in Section \ref{sec:Explicit-basis}.
We also mension some conjecture on the choice of the bases of the
cohomology group for any $z\in Z$. The supporting fact is the holonomicity
and non-existence of the singularity of the system of differential
equations on $Z$ satisfied by $F(z)$. We give an elementary proof
of this fact in Appendix.

\section{General HGF}

In this section we review briefly the definition of the general HGF. 

\subsection{Maximal abelian groups}

Let $N$ be a positive integer and $\l=(n_{1},\dots,n_{\ell})$ be
a partition of $N.$ With the partition $\l$, we associate the maximal
abelian subgroup $H_{\l}$ of $\GL N$:
\begin{equation}
H_{\l}=J(n_{1})\times\cdots\times J(n_{\ell}),\label{eq:hyp-0}
\end{equation}
 where 
\[
J(n):=\left\{ h=\sum_{0\leq i<n}h_{i}\La^{i}\ \mid\ h_{i}\in\C,\ \ h_{0}\neq0\right\} \subset\GL n,
\]
$\Lambda:=\La_{n}=(\de_{i+1,j})_{1\leq i,j\leq n}$ being the shift
matrix of size $n$. In (\ref{eq:hyp-0}), $(h^{(1)},\dots,h^{(\ell)})\in J(n_{1})\times\cdots\times J(n_{\ell})$
is identified with the block diagonal matrix $\diag(h^{(1)},\dots,h^{(\ell)})\in\GL N$.
Note that $\dim_{\C}J(n)=n$, and hence $\dim_{\C}H_{\l}=N$. 
\begin{rem}
\label{rem:hyp-4} 1) $J(n)$ is the centralizer of an element 
\[
C(n,a)=\begin{pmatrix}a & 1\\
 & \ddots & \ddots\\
 &  & \ddots & 1\\
 &  &  & a
\end{pmatrix}\in\GL n.
\]
 2) An element $a\in\GL N$ is called a regular element if the orbit
$O(a)$ of $a$ by the adjoint action is of maximum dimension, namely
$\dim O(a)=N^{2}-N$. $a$ is a regular element if and only if there
is a partition $\l=(n_{1},\dots,n_{\ell})$ of $N$ such that the
Jordan normal form of $a$ is 
\[
a\sim C(n_{1},a_{1})\oplus\cdots\oplus C(n_{\ell},a_{\ell})
\]
with distinct eigenvalues $a_{1},\dots,a_{\ell}$. The centralizer
of this Jordan normal form is $H_{\l}.$ 

3) Let $\C[X]$ be the polynomial ring of an indeterminate $X$ and
let $(X^{n})$ be the ideal of $\C[X]$ generated by $X^{n}$. Then
the Jordan group $J(n)$ is identified with the group of units of
the quotient ring $\C[X]/(X^{n})$ by the correspondence
\[
\sum_{0\leq i<n}h_{i}\La^{i}\mapsto\sum_{0\leq i<n}h_{i}X^{i}.
\]
\end{rem}

\subsection{Character of $H_{\protect\l}$}

The general HGF is defined as the Radon transform of a character of
the universal covering group $\tilde{H}_{\l}$ of $H_{\l}$. So we
describe the characters of $\tilde{H}_{\l}$.

Let $x=(x_{0},x_{1},x_{2},\dots)$ be a sequence of variables and
let $\theta_{m}(x)\ (m\geq0)$ be the function defined by 
\begin{align*}
\sum_{0\leq m<\infty}\theta_{m}(x)T^{m} & =\log(x_{0}+x_{1}T+x_{2}T^{2}+\cdots)\\
 & =\log x_{0}+\log\left(1+\frac{x_{1}}{x_{0}}T+\frac{x_{2}}{x_{0}}T^{2}+\cdots\right)\\
 & =\log x_{0}+\sum_{1\leq k<\infty}\frac{(-1)^{k-1}}{k}\left(\frac{x_{1}}{x_{0}}T+\frac{x_{2}}{x_{0}}T^{2}+\cdots\right)^{k}
\end{align*}
Here $\theta_{0}(x)=\log x_{0}$, and $\theta_{m}(x)$ $(m\geq1)$
is a quasi-homogeneous polynomial of $x_{1}/x_{0},\dots,x_{m}/x_{0}$
of weight $m$ if the weight of $x_{k}/x_{0}$ is defined to be $k$:

\begin{align*}
\theta_{m}(x) & =\sum(-1)^{k_{1}+\cdots+k_{m}-1}\frac{(k_{1}+\cdots+k_{m}-1)!}{k_{1}!\cdots k_{m}!}\left(\frac{x_{1}}{x_{0}}\right)^{k_{1}}\cdots\left(\frac{x_{m}}{x_{0}}\right)^{k_{m}},
\end{align*}
where the sum is taken over the indices $(k_{1},\dots,k_{m})\in\Z_{\geq0}^{m}$
such that $k_{1}+2k_{2}+\cdots+mk_{m}=m$.
\begin{lem}
\cite{GRS88}\label{lemma:hyp-1} The correspondence
\[
h=\sum_{0\leq i<n}h_{i}\La^{i}\mapsto(h_{0},\theta_{1}(h),\dots,\theta_{n-1}(h))
\]
 gives the isomorphism $J(n)\simeq\C^{\times}\times\C^{n-1}$, where
$\C^{n-1}$ is an additive group with the addition of vector space.
\end{lem}

The following result is the consequence of the above lemma.
\begin{lem}
Let $\tilde{J}(n)$ be the universal covering group of $J(n)$. Then
a character $\chi_{n}:\tilde{J}(n)\to\cbatu$ is given by
\[
\chi_{n}(h;\a)=\exp\left(\sum_{0\leq i<n}\a_{i}\theta_{i}(h)\right)=h_{0}^{\a_{0}}\exp\left(\sum_{1\leq i<n}\alpha_{i}\theta_{i}(h)\right)
\]
for some complex constants $\a=(\a_{0},\dots,\a_{n-1})$.
\end{lem}

Since $H_{\l}$ is a product of $J(n_{k})$, a character of the group
$\tH_{\l}$ is a product of the characters $\chi_{n_{k}}$ of $\tilde{J}(n_{k})$. 
\begin{prop}
Let $\chi_{\l}:\tH_{\l}\to\cbatu$ be a character. Then we have

\begin{equation}
\chi_{\l}(h;\a)=\prod_{1\leq k\leq\ell}\chi_{n_{k}}(h^{(k)};\a^{(k)}),\label{eq:hyp-1}
\end{equation}
for some $\alpha=(\alpha^{(1)},\dots,\alpha^{(\ell)})\in\C^{N},$
$\alpha^{(k)}=(\alpha_{0}^{(k)},\alpha_{1}^{(k)},\dots,\alpha_{n_{k}-1}^{(k)})\in\C^{n_{k}}$,
where $h=(h^{(1)},\cdots,h^{(\ell)})\in\tilde{H}_{\l},\;h^{(k)}\in\tilde{J}(n_{k})$. 
\end{prop}

\subsection{Rough sketch of Radon transform}

Next let us consider the ``Radon transform''. Let $r$ be a positive
integer such that $r+1<N$ and consider the following double fibration
\begin{alignat*}{3}
 &  & \mathrm{F}_{1,r+1}(\C^{N})\\
\pi_{1} & \swarrow &  & \searrow\pi_{2}\\
\Ps(\C^{N}) &  &  &  & \mathrm{G}_{r+1}(\C^{N}),
\end{alignat*}
where $\mathrm{F}_{1,r+1}(\C^{N})$ is the flag manifold:
\[
\mathrm{F}_{1,r+1}(\C^{N})=\{(v_{1},v_{2})\mid v_{1}\subset v_{2}\subset\C^{N}:\text{subspaces,}\dim v_{1}=1,\dim v_{2}=r+1\},
\]
$\Ps(\C^{N})$ is the projective space of one-dimensional linear subspaces
of $\C^{N}$, and $\pi_{1},\pi_{2}$ are projections:
\[
\pi_{i}((v_{1},v_{2}))=v_{i},\quad i=1,2.
\]
Roughly speaking, the Radon trasnform is the following. Suppose that
$f$ is a ``function'' on $\Ps(\C^{N})$ and that $\pi_{1}^{*}f$
is its pullback by $\pi_{1}.$ Take $v\in\gras$ and restrict $\pi_{1}^{*}f$
to the fiber $\pi_{2}^{-1}(v)\simeq\Ps(v)$ which is isomorphic to
$\Ps^{r}$, and integrate it on some $r$-dimensional cycle $C$ of
$\pi_{2}^{-1}(v)$. Then we have the ``function'' $v\mapsto\int_{C}(\pi_{1}^{*}f)\vert_{\pi_{2}^{-1}(v)}\cdot\tau$
on $\gras$, where $\tau$ is the $r$-form explained in Definition
\ref{def:hyp-3}.

In defining the general HGF, we will identify $H_{\l}$ with a Zariski
open subset of the space $\C^{N}$ of the homogeneous coordinates
of $\Ps(\C^{N})$ and take a character $\chi_{\l}$ as $f$ in the
above story, which is regarded as a multivalued function on this Zariski
open set of $\C^{N}$. Here $\C^{N}$ is considered as the vector
space of $N$-dimensional \emph{row }vectors. 

Consider the biholomorphic map

\[
\iota:\ \ H_{\l}\to\prod_{1\leq k\leq\ell}\left(\C^{\times}\times\C^{n_{k}-1}\right)\subset\C^{N}
\]
 defined by 
\[
\iota(h)=(h_{0}^{(1)},\dots,h_{n_{1}-1}^{(1)},\dots,h_{0}^{(\ell)},\dots,h_{n_{\ell}-1}^{(\ell)})
\]
for $h=\bigoplus_{k}\left(\sum_{0\leq j<n_{k}}h_{j}^{(k)}\La_{n_{k}}^{j}\right)\in H_{\l}$.
The map $\iota$ can be lifted to the map from $\tilde{H}_{\l}$ to
$\prod_{1\leq k\leq\ell}\left(\tilde{\C}^{\times}\times\C^{n_{k}-1}\right)$.
This lift is also denoted by $\iota.$ By this map, we regard the
character $\chi_{\l}$ as a multivalued holomorphic function on $\iota(H_{\l})\subset\C^{N}$. 

We turn to the description of the fiber $\pi_{2}^{-1}(v)$. Let $\matt{r+1,N}$
denote the set of $(r+1)\times N$ complex matrices of rank $r+1$.
Take $v\in\gras$ and choose $z'_{0},z'_{1},\dots,z'_{r}\in\C^{N}$
so that $v=span_{\C}\{z'_{0},z'_{1},\dots,z'_{r}\}$. If we take another
basis $w'_{0},w'_{1},\dots,w'_{r}$ of $v$, then we have 
\[
z=\left(\begin{array}{c}
z'_{0}\\
\vdots\\
z'_{r}
\end{array}\right),w=\left(\begin{array}{c}
w'_{0}\\
\vdots\\
w'_{r}
\end{array}\right)\in\matt{r+1,N}
\]
which are connected as $w=gz$ by some $g\in\GL{r+1}$. So, by the
left action of $\GL{r+1}$ on $\matt{r+1,N}$, we have the identification
$\GL{r+1}\backslash\matt{r+1,N}\simeq\gras$ by the correspondence
\[
\GL{r+1}\backslash\matt{r+1,N}\ni[z]\mapsto span_{\C}\{z'_{0},z'_{1},\dots,z'_{r}\}\in\gras.
\]
We use freely this identification in what follows. 

Take a point $[z]\in\gras$. Then the fiber $\pi_{2}^{-1}([z])=\Ps([z])$
is the set of one-dimensional subspaces of $span_{\C}\{z'_{0},z'_{1},\dots,z'_{r}\}$
and a basis of the one-dimesional subspace is given by
\[
t_{0}z'_{0}+\cdots+t_{r}z'_{r}=(t_{0},\dots,t_{r})z=tz\in\C^{N}
\]
with a non-zero vector $t=(t_{0,}t_{1},\dots,t_{r})\in\C^{r+1}$.
We regard $t$ as the homogeneous coordinates of $\Ps^{r}\simeq(\C^{r+1}\setminus\{0\})\slash\cbatu$,
then the correspondence $[t]\mapsto span_{\C}\{tz\}$ gives the identification
$\Ps^{r}\simeq\Ps([z])$. The space of homogeneous coordinates $t$
will be denoted by $T=\C^{r+1}$. Note that, if we write $z$ as $z=(z_{1},\dots,z_{N})$
with the column vectors $z_{j}$, $tz=(tz_{1},\dots,tz_{N})$ is a
vector whose $j$-th entry $tz_{j}$ is a linear polynomial of $t$.
Hence $\chi_{\l}(\iota^{-1}(tz);\a)$ is the expression of $\pi_{1}^{*}\chi_{\l}\vert_{\pi_{2}^{-1}([z])}$
in terms of the homogeneous coordinates $t$ of $\Ps^{r}\simeq\pi_{2}^{-1}([z])$.
It is seen from the explicit form (\ref{eq:hyp-1}) of $\chi_{\l}$
that $\chi_{\l}(\iota^{-1}(tz);\a)$, constructed by substituting
linear polynomials $tz=(tz_{1},\dots,tz_{N})$ into the character
$\chi_{\l}(\cdot\,;\a)$, is a homogeneous function of $t$ of degree
$\a_{0}^{(1)}+\cdots+\a_{0}^{(\ell)}$:
\[
\chi_{\l}(\iota^{-1}((ct)z);\a)=c^{\a_{0}^{(1)}+\cdots+\a_{0}^{(\ell)}}\chi_{\l}(\iota^{-1}(tz);\a),\quad\forall c\in\cbatu.
\]
For the character $\chi_{\l}(\cdot\,;\a)$, we assume the conditions
\begin{equation}
\sum_{1\leq k\leq\ell}\alpha_{0}^{(k)}=-r-1\label{eq:hyp-2}
\end{equation}
 and 
\begin{equation}
\a_{n_{k}-1}^{(k)}\begin{cases}
\notin\Z & \text{if}\;n_{k}=1,\\
\neq0 & \text{if}\;n_{k}>1.
\end{cases}\label{eq:hyp-2-1}
\end{equation}
The condition (\ref{eq:hyp-2}) implies that $\chi_{\l}(\iota^{-1}(tz);\a)$
is regarded as a multivalued global section of the sheaf $\cO(-r-1)$
on $\Ps^{r}$.

\subsection{\label{subsec:general-HGF}The general HGF}

For the given partition $\l$, let us define the space of coefficients
$z$ of linear polynomials of $t$, which will serve as the domain
of definition of the general HGF.

We sometimes identify a partition $\l=(n_{1},\dots,n_{\ell})$ with
the Young diagram which is obtained by arraying boxes, $n_{1}$ boxes
in the first row, $n_{2}$ boxes in the second row, and so on, as
is illustrated in Figure 2.1. A sequence $\mu=(m_{1},\dots,m_{\ell})\in\Z_{\geq0}^{\ell}$
is called a \emph{subdiagram} of $\l$ if it satisfies $0\leq m_{k}\leq n_{k}\;(1\leq k\leq\ell)$
and this condition is expressed as $\mu\subset\l$. The sum $|\mu|=m_{1}+\cdots+m_{\ell}$
is called the weight of $\mu$. See Figure 2.1. For a given $z=(z^{(1)},\dots,z^{(\ell)})\in\matt{r+1,N}$
with $z^{(k)}=(z_{0}^{(k)},\dots,z_{n_{k}-1}^{(k)})\in\mat{r+1,n_{k}}$
and for any subdiagram $\mu\subset\l,|\mu|=r+1$, we put 
\[
z_{\mu}=(z_{0}^{(1)},\dots,z_{m_{1}-1}^{(1)},\dots,z_{0}^{(\ell)},\dots,z_{m_{\ell}-1}^{(\ell)})\in\mat{r+1}.
\]

\begin{defn}
\label{def:hyp-2} The subset $Z_{r+1}\subset\mat{r+1,N}$, defined
by
\[
\zrn=\{z\in\mat{r+1,N}\ \mid\ \det z_{\mu}\neq0\ \ \mbox{for any}\ \ \mu\subset\l,|\mu|=r+1\},
\]
is called the \emph{generic stratum} with respect to $H_{\l}$.
\end{defn}

\begin{figure}
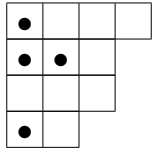

\label{young}
\[
\young(\bullet\;\;\;,\bullet\bullet\;,\;\;\;,\bullet\;)
\]
\caption{Subdiagram $\mu=(1,2,0,1)$ in $\lambda=(4,3,3,2)$}

\end{figure}

In the space $T$, we consider the $r$-form:
\[
\tau=\sum_{0\leq i\leq r}(-1)^{i}dt_{0}\wedge\cdots\wedge\widehat{dt_{i}}\wedge\cdots\wedge dt_{r}.
\]
 In the chart $\{[t]\in\Ps^{r}\mid t_{0}\neq0\}$, we take the affine
coordinates $u_{1},\dots,u_{r}$ by $u_{i}=t_{i}/t_{0}$. Then we
have 
\[
\tau=(t_{0})^{r+1}du_{1}\wedge\cdots\wedge du_{r},
\]
and, by virtue of the assumption (\ref{eq:hyp-2}), we see that 
\[
\chi_{\l}(\iota^{-1}(tz);\a)\cdot\tau=\chi_{\l}(\iota^{-1}(\vec{u}z);\a)du_{1}\wedge\cdots\wedge du_{r},
\]
where $\vec{u}=(1,u_{1},\dots,u_{r})$. In other charts of $\Ps^{r}$,
we have the similar expressions, and so $\chi_{\l}(\iota^{-1}(tz);\a)\cdot\tau$
defines a multivalued holomorphic $r$-form on $\Ps^{r}$ which has
the ramification divisor $\cup_{1\leq j\leq\ell}H_{j}$ consisting
of the hyperplanes $H_{j}:=\{[t]\in\Ps^{r}\mid tz_{0}^{(j)}=0\}$.
\begin{defn}
\label{def:hyp-3} Under the assumptions (\ref{eq:hyp-2}) and (\ref{eq:hyp-2-1}),
the \textit{general HGF of type} $\l$ is the function of $z\in\zrn$
defined by 
\begin{equation}
F(z,\alpha)=\int_{C(z)}\chi_{\l}(\iota^{-1}(tz);\a)\cdot\tau,\label{GHI}
\end{equation}
where $C(z)$ is an $r$-cycle in $\Ps^{r}\setminus\cup_{1\leq k\le\ell}\{tz_{0}^{(k)}=0\}$
of the homology group defined by the integrand $\chi_{\l}(\iota^{-1}(tz);\a)$. 
\end{defn}

In the following, we simply write $tz$ instead of $\iota^{-1}(tz)$
for the sake of simplicity.
\begin{rem*}
In the notation of introducton, $U(u,z)$ in (\ref{eq:intro-4}) is
related to $\chi_{\l}(\vec{u}z;\a)$, and $f_{1},\dots,f_{\ell},g$
have the form
\[
f_{j}=\vec{u}\cdot z_{0}^{(j)},\quad g=\sum_{k=1}^{\ell}\left(\sum_{j=1}^{n_{k}-1}\a_{j}^{(k)}\theta_{j}(\vec{u}z^{(k)})\right).
\]
The condition (\ref{eq:hyp-2-1}) implies that the local system $\Lc$
has the nontrivial monodromy $\exp(2\pi\sqrt{-1}\a_{0}^{(j)})\neq1$
along $H_{j}$ if $n_{j}=1$, and $g$ is a rational function of $u$
having a pole along $H_{j}$ of order $n_{j}-1$ if $n_{j}>1$.
\end{rem*}

\section{\label{sec:Main-theorem}Main theorem}

\subsection{Twisted algebraic de Rham cohomology\label{subsec:deRham-1}}

For $z\in\zrn$ fixed, we see from the assumption (\ref{eq:hyp-2})
on $\a$ that $\chi(tz;\a)$ is a multivalued section of the sheaf
$\cO_{X}(-r-1)$ on 
\[
X=\Ps^{r}\setminus\left(\cup_{1\leq k\leq\ell}H_{k}\right),\quad H_{k}=\{[t]\in\Ps^{r}\mid tz_{0}^{(k)}=0\}.
\]
Take a global section $(tz_{0}^{(1)})^{r+1}$ of $\cO_{X}(r+1)$,
and consider the multivalued holomorphic function $f=\chi(tz,\a)(tz_{0}^{(1)})^{r+1}$
on $X$. Define the twisted differential 
\[
\na_{\w}=f^{-1}\circ d\circ f=d+\w\wedge,\quad\w=df/f.
\]
Here $\w$ is the rational $1$-form holomorphic on $X.$ On the affine
chart $U=\{[t]\in\Ps^{r}\mid t_{0}\neq0\}$ with the affine coordinates
$(x_{1},\dots.x_{r}),x_{i}=t_{i}/t_{0}$, $\w$ is given by
\begin{equation}
\w=\sum_{k=1}^{\ell}\{(\a_{0}^{(k)}+(r+1)\de_{k1})d\log\theta_{0}(\bdx z^{(k)})+\sum_{1\leq j<n_{k}}\a_{j}^{(k)}d\theta_{j}(\bdx z^{(k)})\},\label{eq:deRham-2}
\end{equation}
where $\bdx=(1,x_{1},\dots.x_{r})$. We see that $\w$ has a pole
of order $n_{k}$ on the hyperplane $H_{k}$. Let $\W^{p}(X)$ be
the set of rational $p$-forms having poles at most on $\cup_{1\leq k\leq\ell}H_{k}$.
Then we have the complex

\[
(\W^{\bu}(X),\na_{\w}):\ \ 0\to\W^{0}(X)\overset{\na_{\w}}{\lto}\W^{1}(X)\overset{\na_{\w}}{\lto}\cdots\overset{\na_{\w}}{\lto}\W^{r}(X)\to0
\]
 which we call the (twisted) \emph{algebraic de Rham complex} associated
with the integral (\ref{GHI}).
\begin{defn}
The cohomology groups of the complex $(\W^{\bu}(X),\nabla_{\w})$: 

\[
\cohom pX{\w}:=\ker\{\nabla_{\w}:\W^{p}(X)\to\W^{p+1}(X)\}/\nabla_{\w}\W^{p-1}(X),\quad p=0,1,\dots
\]
are called the \emph{twisted algebraic de Rham cohomology groups}
(simply, de Rham cohomology groups).
\end{defn}

\subsection{Generalized Veronese map}

We introduce the map $\Phi_{\l}:\znn\to\zrn$ called Veronese map
of type $\l$, which will be used in the statement of Theorem \ref{thm:main}.
First, we treat the simple case $\l=(n)$, and then we give the definition
of $\Phi_{\l}$ for general $\l$ in Definition \ref{def:vero-3}. 

\subsubsection{The case $\protect\l=(n)$}

Let $V$ be a complex vector space of $\dim V=2$ and let $R=\C[T]/(T^{n})$
be the quotient ring by the ideal $(T^{n})$ of $\C[T]$ generated
by $T^{n}$. Put $W:=V\otimes_{\C}R$; then $W$ is a free $R$-module
of rank $2$ as well as a left $\mathrm{GL}(V)$-module by the action
$g\cdot(v\otimes h)=(gv)\otimes h,$ where $g\in\mathrm{GL}(V),v\in V$
and $h\in R.$ The module $W$ and the ring $R$ above are also denoted
as $W_{n}$ and $R_{n}$, respectively, when it is necessary to emphasize
their dependence on $n$.

Let $S^{r}(W)$ be the symmetric tensor product of $W$ as $R$-module.
Since $S^{r}(W)\simeq S^{r}(V)\otimes_{\C}R,$ $S^{r}(W)$ is a free
$R$-module of rank $r+1.$ Also $S^{r}(W)$ is endowed with the structure
of left $\mathrm{GL}(V)$-module induced from that for $W.$
\begin{defn}
The $(\mathrm{GL}(V),R_{n})$-equivariant map $\Phi:W\to S^{r}(W)$
is defined by
\[
w\mapsto\Phi(w)=\overbrace{w\otimes\cdots\otimes w}^{r},
\]
and is called\emph{ }\textit{\emph{the }}\textit{(generalized) Veronese
map}\textit{\emph{.}}\textit{ }\textit{\emph{Sometimes we write this
map $\Phi$ as $\Phi_{n}$ in order to emphasize the dependence on}}\textit{
$n.$} 
\end{defn}

Let us write down the Veronese map $\Phi_{n}$ using the $\C$-bases
of $W$ and $S^{r}(W)$. Let $e_{0},e_{1}$ be a basis of $V$ by
which we identify $V$ with $\C^{2}.$ Since $W=V\otimes R_{n}$,
using its $\C$-basis $e_{i}\otimes T^{j}\;(i=0,1;0\leq j<n)$, we
can identify $W$ with $\mat{2,n}$ as a $\C$-vector space by the
correspondence 
\[
W\ni w=\sum_{i=0,1}\sum_{0\leq j<n}w_{ij}e_{i}\otimes T^{j}\mapsto(w_{ij})\in\mat{2,n}.
\]
Similarly we identify $S^{r}(W)$ with $\mat{r+1,n}$ as a $\C$-vector
space. For this, take a basis $\bde_{0},\dots,\bde_{r}$ of $S^{r}(V)$
defined by 
\[
\bde_{k}=\sum_{i_{1}+\cdots+i_{r}=k}e_{i_{1}}\otimes\cdots\otimes e_{i_{r}},
\]
and identify $S^{r}(V)$ with $\C^{r+1}$. Hence, using the basis
$\bde_{k}\otimes T^{j}$ of $S^{r}(W)\simeq S^{r}(V)\otimes R_{n}$,
we can identify $S^{r}(W)$ with $\mat{r+1,n}$ by the correspondence

\[
S^{r}(W)\ni z=\sum_{0\leq i\leq r}\sum_{0\leq j<n}z_{ij}\bde_{i}\otimes T^{j}\mapsto(z_{ij})\in\mat{r+1,n}.
\]
 For $w\in W,$ we put

\[
\Phi_{n}(w)=\overbrace{w\otimes\cdots\otimes w}^{r}=\sum_{0\leq i\leq r}\sum_{0\leq j<n}\varphi_{ij}(w)\bde_{i}\otimes T^{j}.
\]
 It is easily seen that the explicit form of the polynomials $\varphi_{ij}(w)$
is given by 
\[
\varphi_{ij}(w)=\sum w_{0j_{1}}\cdots w_{0j_{r-i}}w_{1j_{r-i+1}}\cdots w_{1j_{r}},
\]
 where the sum is taken over all the indices $(j_{1},\dots,j_{r})$
satisfying $0\leq j_{p}<n$ and $j_{1}+\cdots+j_{r}=j.$ Therefore
$\varphi_{ij}(w)$ are homogeneous polynomials of $w\in\mat{2,n}$
of degree $r$. Then the map
\[
\mat{2,n}\ni w\mapsto(\varphi_{ij}(w))\in\mat{r+1,n}
\]
 gives the expression of the map $\Phi_{n}$ with respect to the $\C$-bases
of $W$ and $S^{r}(W)$.
\begin{rem}
When $n=1,$ the map $\Phi$ is given, in terms of coordinates, by
\[
\C^{2}\ni{}^{t}(w_{0},w_{1})\mapsto{}^{t}(w_{0}^{r},w_{0}^{r-1}w_{1},\dots,w_{0}w_{1}^{r-1},w_{1}^{r})\in\C^{r+1}
\]
and induces the map $\Ps^{1}\to\Ps^{r}$ which coincides with the
Veronese embedding in the usual sense. 
\end{rem}

The reason why we introduced the Veronese map $\Phi_{n}$ is that
Proposition \ref{prop:vero-7} below holds. Take $z=(z_{0},\dots,z_{n-1})\in\mat{2,n}$
and define the polynomials $l_{k}(x)$ in $x$ by 
\[
l_{k}(x):=(1,x)z_{k}=z_{0k}+xz_{1k},\quad0\leq k<n.
\]
Let $x_{1},\dots,x_{r}$ be $r$ copies of the variable $x$ and let
$E_{i}$ be the $i$-th elementary symmetric function of $x_{1},\dots,x_{r}.$
Taking another variables $y=(y_{1},\dots,y_{r})$ and put $\bdy=(1,y_{1},\dots y_{r})$.
The following proposition follows from the definition of the Veronese
map $\Phi=\Phi_{n}$.
\begin{prop}
\label{prop:vero-7} Let $L_{j}(y)\;(0\leq j<n)$ be the polynomials
of degree $1$ defined by 
\[
L_{j}(y)=\bdy\Phi_{n}(z)_{j},
\]
where $\Phi_{n}(z)_{j}$ is the $j$-th column vector of the Veronese
image $\Phi_{n}(z)$ of $z\in\mat{2,n}$. If the variables $y$ are
connected with $x_{1},\dots,x_{n}$ by $y_{i}=E_{i}(x_{1},\dots,x_{n})$,
then we have the identity 
\[
\sum_{0\leq j<n}L_{j}(y)T^{j}\equiv\prod_{1\leq i\leq r}\{l_{0}(x_{i})+l_{1}(x_{i})T+\cdots+l_{n-1}(x_{i})T^{n-1}\}\quad\mathrm{mod}\,(T^{n}).
\]
\end{prop}

\subsubsection{The case $\protect\l=(n_{1},\dots,n_{\ell})$}

We introduce $\Phi_{\l}$ for the partition $\l=(n_{1},\dots,n_{\ell})$
of $N$. Put $W_{\l}=W_{n_{1}}\oplus\cdots\oplus W_{n_{\ell}}$ and
$R_{\l}=R_{n_{1}}\times\cdots\times R_{n_{\ell}}$, where $W_{n}=V\otimes R_{n}$.
Then $W_{\l}$ is considered as a left $\mathrm{GL}(V)$-module as
well as an $R_{\l}$-module.
\begin{defn}
\label{def:vero-3}Define the map 
\[
\Phi_{\l}:W_{\l}\to\oplus_{k}S^{r}(W_{n_{k}})
\]
 by 
\[
\bigl(w^{(1)},\dots,w^{(\ell)}\bigr)\mapsto\bigl(\Phi_{n_{1}}(w^{(1)}),\dots,\Phi_{n_{\ell}}(w^{(\ell)})\bigr).
\]
 This is called the \textit{Veronese map of type $\l$. } 
\end{defn}

Similarly to the case $\Phi_{n}$, the Veronese map $\Phi_{\l}$ induces
the map $\mat{2,N}\to\mat{r+1,N}$ which will also be denoted by $\Phi_{\l}.$
Let $R_{n}^{\times}$ be the group of units in $R_{n}$. Then $R_{\l}^{\times}:=R_{n_{1}}^{\times}\times\cdots\times R_{n_{\ell}}^{\times}$
is the group of units of $R_{\l}$ which acts on $W_{\l}$ in an obvious
way. This action is interpreted as the action of $H_{\l}=J(n_{1})\times\cdots\times J(n_{\ell})$
on $\mat{2,N}$, see Remark \ref{rem:hyp-4}. 

Recall that $Z_{r+1}\subset\matt{r+1,N}$ denotes the generic stratum
with respect to $H_{\l}$ defined in Definition \ref{def:hyp-2}.
\begin{prop}
(\cite{kim2010}, Proposition 5.6) \label{prop:vero-5} The polynomial
map $\Phi_{\l}:\mat{2,N}\to\mat{r+1,N}$ takes the generic stratum
$\znn$ into $\zrn$. 
\end{prop}

The image $\Phi_{\l}(Z_{2})\subset Z_{r+1}$ is called the \emph{Veronese
image}, and its point is called a \emph{Veronese point}.

\subsection{\label{subsec:Main}Main Theorem}

Let $z\in\znn$ and let $w=\Phi_{\l}(z)\in Z_{r+1}$ be the corresponding
Veronese point. We put 
\begin{equation}
X=\Ps^{1}\setminus\cup_{1\leq k\leq\ell}\{[\bdx]\in\Ps^{1}\mid\bdx z_{0}^{(k)}=0\},\quad Y=\Ps^{r}\setminus\cup_{1\leq k\leq\ell}\{[\bdy]\in\Ps^{r}\mid\bdy w_{0}^{(k)}=0\}.\label{eq:main-0}
\end{equation}
 On $X,$ we have the de Rham complex $(\W^{\bu}(X),\na_{\w})$ with 

\begin{equation}
\w=\sum_{1\leq k\leq\ell}\sum_{0\leq i<n_{k}}\a_{i}^{(k)}d\theta_{i}(\bdx z^{(k)}),\label{eq:main-5}
\end{equation}
 and on $Y$, the de Rham complex $(\W^{\bu}(Y),\na_{\tilde{\w}})$
with 

\begin{equation}
\tilde{\w}=\sum_{1\leq k\leq\ell}\sum_{0\leq i<n_{k}}\a_{i}^{(k)}d\theta_{i}(\bdy w^{(k)}).\label{eq:main-6}
\end{equation}
 We assume that the parameter $\a\in\C^{N}$ satisfies 

\begin{equation}
\a_{0}^{(1)}+\cdots+\a_{0}^{(\ell)}=0,\label{eq:main-1}
\end{equation}
and

\begin{equation}
\a_{n_{k}-1}^{(k)}\begin{cases}
\notin\Z & \text{if}\,n_{k}=1,\\
\neq0 & \text{if}\,n_{k}>1.
\end{cases}\label{eq:main-2}
\end{equation}

Note that we defined the $1$-form $\w$ in Section \ref{subsec:deRham-1}
as (\ref{eq:deRham-2}) for the parameter $\a$ appeared in the definition
of the integral (\ref{GHI}). The parameter $\a$ in this section
is not the same as that of Section \ref{subsec:general-HGF}. Here
we renamed $\a_{0}^{(1)}+2$ as $\a_{0}^{(1)}$, see (\ref{eq:deRham-2}).
Hence the condition (\ref{eq:main-1}) corresponds to (\ref{eq:hyp-2}). 

The following result is about the de Rham cohomology group $\cohom{\bu}X{\w}$. 
\begin{prop}
\label{prop:main-3}(\cite{kim97-1}) For $z\in\znn$, and $\a\in\C^{N}$satisfying
(\ref{eq:main-1}) and (\ref{eq:main-2}), we have 
\end{prop}

\begin{enumerate}
\item $\cohom pX{\w}=0$ for $p\neq1$,
\item $\dim_{\C}\cohom 1X{\w}=N-2.$
\end{enumerate}
The following theorem is the main result of this paper which describe
the exterior power structure of the cohomology group $\cohom{\bu}Y{\tilde{\w}}$
for a Veronese point.
\begin{thm}
\label{thm:main}Let $z,w,X,Y$ be as above and assume that $\a$
satisfy the conditions (\ref{eq:main-1}) and (\ref{eq:main-2}).
Then 

(1) $Y$ is the symmetric product $S^{r}X$ of $r$ copies of $X$
and 
\[
H^{p}(\W^{\bu}(Y),\na_{\tilde{\w}})\simeq\begin{cases}
\wedge^{r}H^{1}(\W^{\bu}(X),\na_{\w}) & \mbox{if }p=r,\\
0 & \mbox{otherwise},
\end{cases}
\]

(2) $\dim_{\C}H^{r}(\W^{\bu}(Y),\na_{\tilde{\w}})=\binom{N-2}{r}$. 
\end{thm}

Note that the second assertion of Theorem \ref{thm:main} is a consequence
of the first assertion and Proposition \ref{prop:main-3}. We mention
that the construction of the isomorphism in Theorem \ref{thm:main}
is important in obtaining the explicit form of the bases of the cohomology
group $H^{r}(\W^{\bu}(Y),\na_{\tilde{\w}})$. See Proposition \ref{thm:ext-3}.

\section{\label{sec:Proof-of-theorem}Proof of theorem \ref{thm:main}}

\subsection{K\"unneth formula}

Let $X^{r}$ be the product of $r$ copies of $X$, and let $\pi_{i}:X^{r}\to X$
be the projection to the $i$-th factor. We denote the external product
of the twisted differential $\na_{\w}$ on $X$ by $\boxtimes^{r}\na_{\w}$,
namely,
\[
\boxtimes^{r}\na_{\w}=\pi_{1}^{*}\na_{\w}\otimes\cdots\otimes\pi_{r}^{*}\na_{\w}.
\]

\begin{lem}
The map

\[
\kappa:\otimes^{r}\cohom 1X{\w}\lto\cohomr
\]
 defined by
\[
\kappa(\f_{1}\otimes\cdots\otimes\f_{r})=\pi_{1}^{*}\f_{1}\wedge\cdots\wedge\pi_{r}^{*}\f_{r}
\]
 is an isomorphism of $\C$-vector spaces.
\end{lem}

\begin{proof}
This is a consequence of the K\"unneth formula. 
\end{proof}
The symmetric group $\mathfrak{S}_{r}$ acts on $X^{r}$ by 
\[
\s\cdot(x_{1},\dots,x_{r})=(x_{\s^{-1}(1)},\dots,x_{\s^{-1}(r)}).
\]
The quotient space $S^{r}X:=\mathfrak{S}_{r}\backslash X^{r}$ is
the $r$-th symmetric product of $X$. Let $\pi:X^{r}\to S^{r}X$
be the natural projection. We prove the first part of the theorem. 
\begin{prop}
Let $w=\Phi_{\l}(z)\in\zrn$ be the Veronese image of $z\in\znn$
and let $Y=\Ps^{r}\setminus\cup_{1\leq k\leq\ell}\{[t]\in\Ps^{r}\mid\,tw_{0}^{(k)}=0\}$.
Then $Y=S^{r}X$.
\end{prop}

\begin{proof}
Let $V$ be a $2$-dimensional vector space over $\C$ and $V^{*}$
be its dual. Let $S^{r}(V)$ and $S^{r}(V^{*})$ be the $r$-th symmetric
tensor products of $V$ and $V^{*}$, respectively. Recall that $S^{r}(V^{*})$
is regarded as the dual of $S^{r}(V)$ by the dual pairing which comes
from that between $V^{\otimes r}$ and $(V^{*})^{\otimes r}$:
\[
\la v_{1}\otimes\cdots\otimes v_{r},v_{1}^{*}\otimes\cdots\otimes v_{r}^{*}\ra=\prod_{1\leq i\leq r}\la v_{i},v_{i}^{*}\ra.
\]
 Let $e_{0},e_{1}$ be a basis of $V$ and let $e_{0}^{*},e_{1}^{*}$
be its dual basis. For $z\in\znn$, put $u_{k}^{*}=z_{00}^{(k)}\est 0+z_{10}^{(k)}\est 1\;(k=1,\dots,\ell)$,
and let $(u_{k}^{*})^{\perp}=\{v\in V\mid\la v,u_{k}^{*}\ra=0\}$.
Then we can write as $X=\Ps(V)\setminus\cup_{1\leq k\leq\ell}(u_{k}^{*})^{\perp}$.
Then the symmetric product $S^{r}X$ can be described as follows.
Let $\cS:V^{\otimes r}\to S^{r}(V)$ be the symmetrization, i.e.,
it is defined by extending linearly

\[
\cS(v_{1}\otimes\cdots\otimes v_{r})=\frac{1}{r!}\sum_{\s\in\fraks_{r}}v_{\s^{-1}(1)}\otimes\cdots\otimes v_{\s^{-1}(r)}.
\]
 Take $v=v_{1}\otimes\cdots\otimes v_{r}=(s_{0}^{1}e_{0}+s_{1}^{1}e_{1})\otimes\cdots\otimes(s_{0}^{r}e_{0}+s_{1}^{r}e_{1})$.
Then 

\[
\cS(v)=\sum_{k=0}^{r}\left(\sum_{j_{1}+\cdots+j_{r}=k}s_{j_{_{1}}}^{1}\cdots s_{j_{r}}^{r}\right)\bde_{k}',\quad\bde_{k}'=\cS(e_{0}^{\otimes(r-k)}\otimes e_{1}^{\otimes k}).
\]
 Thus the map $\cS:V^{\otimes r}\to S^{r}V$ is written in terms of
coordinates as 
\[
t_{k}=\sum_{j_{1}+\cdots+j_{r}=k}s_{j_{_{1}}}^{1}\cdots s_{j_{r}}^{r}.
\]
 Since $v_{1}\otimes\cdots\otimes v_{r}\in\otimes^{r}(V\setminus\cup_{1\leq k\leq\ell}(u_{k}^{*})^{\perp})$
is described by the condition 

\[
\la v_{i},u_{k}^{*}\ra\neq0\quad(1\leq i\leq r,1\leq k\leq\ell),
\]
 it follows that $\la\cS(v_{1}\otimes\cdots\otimes v_{r}),u_{k}^{*}\otimes\cdots\otimes u_{k}^{*}\ra\neq0$.
The left hand side can be written as

\begin{multline*}
\la\cS(v_{1}\otimes\cdots\otimes v_{r}),u_{k}^{*}\otimes\cdots\otimes u_{k}^{*}\ra\\
=t_{0}(z_{00}^{(k)})^{r}+t_{1}(z_{00}^{(k)})^{r-1}z_{10}^{(k)}+\cdots+t_{r}(z_{10}^{(k)})^{r}=tw_{0}^{(k)}.
\end{multline*}
Thus we have 
\[
S^{r}(V\setminus\cup_{1\leq k\leq\ell}(u_{k}^{*})^{\perp})=S^{r}(V)\setminus\cup_{1\leq k\leq\ell}((u_{k}^{*})^{\otimes r})^{\perp}
\]
 and in terms of coordinates

\begin{align*}
S^{r}X & =\Ps(S^{r}(V\setminus\cup_{1\leq k\leq\ell}(u_{k}^{*})^{\perp}))\simeq\Ps(S^{r}(V))\setminus\cup_{1\leq k\leq\ell}((u_{k}^{*})^{\otimes r})^{\perp}\\
 & =\Ps^{r}\setminus\cup_{1\leq k\leq\ell}\{t\in\Ps^{r}\mid tw_{0}^{(k)}=0\}=Y.
\end{align*}
\end{proof}
In terms of coordinates the identification $S^{r}X=Y$ can be given
as follows. Take an affine coordinates $x_{i}=s_{1}^{i}/s_{0}^{i}\;(1\leq i\leq r)$
of $\overbrace{\Ps^{1}\times\cdots\times\Ps^{1}}^{r}$ and take an
affine coordinates of $\Ps^{r}$ as $y_{i}=t_{i}/t_{0}\;(1\leq i\leq r)$
. Then the identifying map can be expressed as $y_{i}=$$i$-th elementary
symmetric function of $x_{1},\dots,x_{r}.$

The action of $\fraks_{r}$ on $X^{r}$ induces its action on the
space $\W^{p}(X^{r})$ of differential forms. Let $\W^{p}(X^{r})^{\fraks_{r}}$
denote the set of elements of $\W^{p}(X^{r})$ invariant by the action
of $\fraks_{r}$. Then, in a similar way as in \cite{iwa-kita}, we
can show the following.
\begin{lem}
Let $\pi:X^{r}\to Y:=S^{r}X$ be the natural projection. Then the
map $\pi^{*}:\W^{p}(Y)\to\W^{p}(X^{r})$ induces an isomorphism $\pi^{*}:\W^{p}(Y)\to\W^{p}(X^{r}){}^{\fraks_{r}}$.
\end{lem}

Let $\tilde{\w}$ be as given in (\ref{eq:main-6}). 
\begin{lem}
$\pi^{*}$induces a map

\[
\pi^{*}:\cohom{\bu}Y{\tilde{\w}}\lto H^{\bu}(\W^{\bu}(X^{r}),\boxtimes^{r}\na_{\w})^{\fraks_{r}}.
\]
 
\end{lem}

\begin{proof}
It is to be shown that $\pi^{*}:\W^{p}(Y)\to\W^{p}(X^{r})$ defines
a cochain map from the complex $(\W^{\bu}(Y),\na_{\tilde{\w}})$ to
$(\W^{\bu}(X^{r}),\boxtimes^{r}\na_{\w})$. So we show $\pi^{*}\circ\na_{\tilde{\w}}=(\boxtimes^{r}\na_{\w})\circ\pi^{*}$.
Take $\eta\in\W^{p}(Y)$. Then

\[
\pi^{*}\circ\na_{\tilde{\w}}\eta=\pi^{*}(d\eta+\tilde{\w}\wedge\eta)=d(\pi^{*}\eta)+\pi^{*}\tilde{\w}\wedge\pi^{*}\eta.
\]
 Thus it is sufficient to show that $\pi^{*}\tilde{\w}=\sum_{1\leq i\leq r}\pi_{i}^{*}\w$.
This holds because 

\begin{align*}
\pi^{*}\tilde{\w} & =\pi^{*}d\log\chi(tw;\a)=d\log\prod_{1\leq i\leq r}\chi(s^{i}z;\a)\\
 & =\sum_{1\leq i\leq r}d\log\chi(s^{i}z;\a)=\sum_{1\leq i\leq r}\pi_{i}^{*}\w.
\end{align*}
 Since $\pi^{*}\eta$ is invariant under the action of $\fraks_{r}$,
the cohomology class $[\pi^{*}\eta]\in H^{\bu}(\W^{\bu}(X^{r}),\boxtimes^{r}\na_{\w})$
is $\fraks_{r}$-invariant.
\end{proof}

\subsection{Transfer}
\begin{defn}
For $\f\in\W^{\bu}(X^{r})$, put 
\[
\mu(\f)=\frac{1}{r!}\sum_{\s\in\fraks_{r}}\s^{*}\f.
\]
\end{defn}

\begin{lem}
\label{lem:proof-4}(1) $\mu$ is a projection from $\W^{\bu}(X^{r})$
to the subspace $\W^{\bu}(X^{r})^{\fraks_{r}}$. 

(2) $\mu$ defines a chain map from $(\W^{\bu}(X^{r}),\boxtimes^{r}\na_{\w})$
to itself and hence it induces a homomorphism $\mu^{*}:H^{\bu}(\W^{\bu}(X^{r}),\boxtimes^{r}\na_{\w})\to H^{\bu}(\W^{\bu}(X^{r}),\boxtimes^{r}\na_{\w})$
whose image is contained in $H^{\bu}(\W^{\bu}(X^{r}),\boxtimes^{r}\na_{\w})^{\fraks_{r}}$.
\end{lem}

\begin{proof}
To show the assertion (1), it is sufficient to show that $\mu(\W^{\bu}(X^{r}))\subset\W^{\bu}(X^{r})^{\fraks_{r}}$
and $\mu|_{\W^{\bu}(X^{r})^{\fraks_{r}}}=id.$ Let $\f\in\W^{\bu}(X^{r})$,
then

\begin{align*}
\s^{*}(\mu(\f)) & =\s^{*}\frac{1}{r!}\sum_{\tau\in\fraks_{r}}\tau^{*}\f=\frac{1}{r!}\sum_{\tau\in\fraks_{r}}(\tau\s)^{*}\f\\
 & =\frac{1}{r!}\sum_{\tau\in\fraks_{r}}\tau^{*}\f=\mu(\f).
\end{align*}
This shows $\mu(\W^{\bu}(X^{r}))\subset\W^{p}(X^{r}){}^{\fraks_{r}}$
and the same computation shows $\mu|_{\W^{\bu}(X^{r})^{\fraks_{r}}}=id.$
To prove (2), we show $\mu\circ(\boxtimes^{r}\na_{\w})=(\boxtimes^{r}\na_{\w})\circ\mu$.
Take $\f\in\W^{\bu}(X^{r})$ and $\s\in\fraks_{r}$. Then

\begin{align*}
\s^{*}\circ(\boxtimes^{r}\na_{\w})\f & =\s^{*}\left(d\f+(\sum_{i}\pi_{i}^{*}\w)\wedge\f\right)\\
 & =d\s^{*}\f+\s^{*}(\sum_{i}\pi_{i}^{*}\w)\wedge\s^{*}\f\\
 & =(\boxtimes^{r}\na_{\w})\circ\s^{*}(\f)
\end{align*}
 since $\s^{*}(\sum_{i}\pi_{i}^{*}\w)=\sum\pi_{i}^{*}\w$.
\end{proof}
\begin{defn}
By the commutative diagram 

\begin{equation}
\xymatrix{\W^{\bu}(X^{r})\ar[dr]^{tf}\ar[d]_{\mu}\\
\W^{\bu}(X^{r})^{\fraks_{t}} & \W^{\bu}(Y)\ar[l]_{\pi^{*}}
}
\label{eq:proof-1}
\end{equation}
we define the map $tf$ which is called a transfer.
\end{defn}

Since $\mu$ and $\pi^{*}$are chain maps, $tf$ is also a chain map
and induces a homomorphism 
\[
(tf)^{*}:H^{\bu}(\W^{\bu}(X^{r}),\boxtimes^{r}\na_{\w})\to H^{\bu}(\W^{\bu}(Y),\na_{\tilde{\w}}).
\]
 We want to show the following.
\begin{prop}
\label{prop:proof-6}$\pi:X^{r}\to Y$ induces the isomorphism

\begin{equation}
\pi^{*}:H^{\bu}(\W^{\bu}(Y),\na_{\tilde{\w}})\to H^{\bu}(\W^{\bu}(X^{r}),\boxtimes^{r}\na_{\w})^{\fraks_{r}}\label{eq:proof-2}
\end{equation}
 
\end{prop}

\begin{proof}
We see from Lemma \ref{lem:proof-4} that the restriction of $\mu^{*}$
to $H^{\bu}(\W^{\bu}(X^{r}),\boxtimes^{r}\na_{\w})^{\fraks_{r}}$
is the identity map. From the commutative diagram \ref{eq:proof-1}
we have 
\[
\pi^{*}\circ tf=\mu,\quad tf\circ\pi^{*}=id.
\]
 Passing to the cohomology level we have 
\[
\pi^{*}\circ(tf)^{*}=\mu^{*},\quad(tf)^{*}\circ\pi^{*}=id.
\]
 Since $\mu^{*}$ is identity on $H^{\bu}(\W^{\bu}(X^{r}),\boxtimes^{r}\na_{\w})^{\fraks_{r}}$,
$(tf)^{*}$ is the inverse homomorphism of $\pi^{*}:H^{\bu}(\W^{\bu}(Y),\na_{\tilde{\w}})\to H^{\bu}(\W^{\bu}(X^{r}),\boxtimes^{r}\na_{\w})^{\fraks_{r}}$.
\end{proof}

\subsection{Exterior power structure}

By virtue of the K\"unneth formula we have the isomorphism $\kappa:\otimes^{r}\cohom 1X{\w}\to\cohomr$,
where 
\[
\kappa(\f_{1}\otimes\cdots\otimes\f_{r})=\pi_{1}^{*}\f_{1}\wedge\cdots\wedge\pi_{r}^{*}\f_{r}.
\]
 
\begin{prop}
We have the isomorphism

\begin{equation}
\kappa:\bigwedge{}^{r}\cohom 1X{\w}\to\cohomr^{\fraks_{r}}.\label{eq:proof-3}
\end{equation}
\end{prop}

\begin{proof}
Let $\f_{i}=g_{i}(x)dx.$ Then $\pi_{i}^{*}\f_{i}=g_{i}(x_{i})dx_{i}$.
The symmetric group $\fraks_{r}$ acts on $\cohomr$ as follows. For
$\s\in\fraks_{r},$

\begin{align*}
\s^{*}(\pi_{1}^{*}\f_{1}\wedge\cdots\wedge\pi_{r}^{*}\f_{r}) & =\s^{*}(g_{1}(x_{1})dx_{1}\wedge\cdots\wedge g_{r}(x_{r})dx_{r})\\
 & =g_{1}(x_{\s^{-1}(1)})\cdots g_{r}(x_{\s^{-1}(r)})dx_{\s^{-1}(1)}\wedge\cdots\wedge dx_{\s^{-1}(r)}\\
 & =\sgn{\s}g_{\s(1)}(x_{1})\cdots g_{\s(r)}(x_{r})dx_{1}\wedge\cdots\wedge dx_{r}.
\end{align*}
 On the other hand, we define the action of $\s$ on $\cohomr$ by

\[
\s^{\#}(\f_{1}\otimes\cdots\otimes\f_{r})=\sgn{\s}\f_{\s(1)}\otimes\cdots\otimes\f_{\s(r)}.
\]
 Then we have the commutative diagram 

\[
\xymatrix{\otimes^{r}H^{1}(\W^{\bu}(X),\na_{\w})\ar[d]_{\s^{\#}}\ar[r]^{\kappa} & H^{r}(\W^{\bu}(X^{r}),\boxtimes^{r}\na_{\w})\ar[d]_{\s^{*}}\\
\otimes^{r}H^{1}(\W^{\bu}(X),\na_{\w})\ar[r]^{\kappa} & H^{r}(\W^{\bu}(X^{r}),\boxtimes^{r}\na_{\w}).
}
\]
Thus if we put $\mu^{\#}=\frac{1}{r!}\sum_{\s\in\fraks_{r}}\s^{\#}$,
then 
\begin{align*}
\kappa\left(\bigwedge{}^{r}H^{1}(\W^{\bu}(X),\na_{\w})\right) & =\kappa\circ\mu^{\#}\left(\otimes^{r}H^{1}(\W^{\bu}(X),\na_{\w})\right)\\
 & =\mu^{*}\circ\kappa\left(\otimes^{r}H^{1}(\W^{\bu}(X),\na_{\w})\right)\\
 & =\mu^{*}\cohomr\\
 & =\cohomr^{\fraks_{r}}.
\end{align*}
This proves the proposition. 
\end{proof}
We write down explicitly the isomorphism
\[
(\pi^{*})^{-1}\circ\kappa:\bigwedge{}^{r}H^{1}(\W^{\bu}(X),\na_{\w})\to H^{\bu}(\W^{\bu}(Y),\na_{\tilde{\w}}).
\]

\begin{prop}
\label{thm:ext-3}The isomorphism $(\pi^{*})^{-1}\circ\kappa$ is
given by
\[
\f_{1}\square\cdots\square\f_{r}\mapsto(tf)^{*}(\pi_{1}^{*}\f_{1}\wedge\cdots\wedge\pi_{r}^{*}\f_{r}),
\]
 where $\square$ denotes the exterior product in $\bigwedge{}^{r}H^{1}(\W^{\bu}(X),\na_{\w})$.
\end{prop}

\begin{proof}
We have 
\begin{align*}
\kappa(\f_{1}\square\cdots\square\f_{r}) & =\kappa\left(\frac{1}{r!}\sum_{\s\in\fraks_{r}}\sgn{\s}\f_{\s(1)}\otimes\cdots\otimes\f_{\s(r)}\right)\\
 & =\kappa\left(\frac{1}{r!}\sum_{\s\in\fraks_{r}}\s^{\#}(\f_{1}\otimes\cdots\otimes\f_{r})\right)\\
 & =\kappa\circ\mu^{\#}(\f_{1}\otimes\cdots\otimes\f_{r})\\
 & =\mu^{*}\circ\kappa(\f_{1}\otimes\cdots\otimes\f_{r})\\
 & =\mu^{*}(\pi_{1}^{*}\f_{1}\wedge\cdots\wedge\pi_{r}^{*}\f_{r}).
\end{align*}
 Thus 

\begin{align*}
(\pi^{*})^{-1}\circ\kappa(\f_{1}\square\cdots\square\f_{r}) & =(\pi^{*})^{-1}\circ\mu^{*}(\pi_{1}^{*}\f_{1}\wedge\cdots\wedge\pi_{r}^{*}\f_{r})\\
 & =(tf)^{*}(\pi_{1}^{*}\f_{1}\wedge\cdots\wedge\pi_{r}^{*}\f_{r}).
\end{align*}
 This proves the proposition and we have finished the proof of Theorem
\ref{thm:main}.
\end{proof}
In constructing a basis of $H^{r}(\W^{\bullet}(Y),\na_{\tilde{\w}})$
explicitly, the following result is helpful.

Suppose that the $1$-forms $\f_{1},\dots,\f_{N-2}$ give a basis
of $H^{1}(\W^{\bullet}(X),\na_{\w})$. And let $\psi_{1},\dots,\psi_{N-2}\in\W^{1}(Y)$
be such that 
\begin{equation}
\pi^{*}\psi_{j}=\sum_{1\leq i\leq r}\pi_{i}^{*}\f_{j}=\sum_{1\leq i\leq r}\f_{j}(x_{i}).\label{eq:proof-4}
\end{equation}
Then we have
\begin{prop}
\label{prop:exp-basis-1}$H^{r}(\W^{\bullet}(Y),\na_{\tilde{\w}})=\bigoplus_{1\leq j_{1}<\cdots<j_{r}\le N-2}\C\psi_{j_{1}}\wedge\cdots\wedge\psi_{j_{r}}.$
\end{prop}

\begin{proof}
Recall that Proposition \ref{thm:ext-3} says that the isomorphism
\[
(\pi^{*})^{-1}\circ\kappa:\bigwedge{}^{r}H^{1}(\W^{\bullet}(X),\na_{\w})\to H^{\bu}(\W^{\bu}(Y),\na_{\tilde{\w}})
\]
 is given by

\[
\f_{1}\square\cdots\square\f_{r}\mapsto(tf)^{*}(\pi_{1}^{*}\f_{1}\wedge\cdots\wedge\pi_{r}^{*}\f_{r}),
\]
where $(tf)^{*}=(\pi^{*})^{-1}\circ\mu^{*}$ and, $\mu^{*}$ and $\pi{}^{*}$
are those defined in Lemma \ref{lem:proof-4} and  Proposition \ref{prop:proof-6},
respectively. We see that 
\begin{align*}
\mu^{*}(\pi_{1}^{*}\f_{j_{1}}\wedge\cdots\wedge\pi_{r}^{*}\f_{j_{r}}) & =\mu^{*}(\f_{j_{1}}(x_{1})\wedge\cdots\wedge\f_{j_{r}}(x_{r}))\\
 & =\frac{1}{r!}\sum_{\s\in\fraks_{r}}\f_{j_{1}}(x_{\s^{-1}(1)})\wedge\cdots\wedge\f_{j_{r}}(x_{\s^{-1}(r)})\\
 & =\frac{1}{r!}\sum_{\s\in\fraks_{r}}\sgn{\s}\cdot\f_{j_{\s(1)}}(x_{1})\wedge\cdots\wedge\f_{j_{\s(r)}}(x_{r}).
\end{align*}
 On the other hand, noting that $\pi^{*}\psi_{j}=\sum_{1\leq i\leq r}\f_{j}(x_{i})$,
we have

\begin{align*}
\pi^{*}(\psi_{j_{1}}\wedge\cdots\wedge\psi_{j_{r}}) & =\left(\sum_{i_{1}}\f_{j_{1}}(x_{i_{1}})\right)\wedge\cdots\wedge\left(\sum_{i_{r}}\f_{j_{r}}(x_{i_{r}})\right)\\
 & =\sum_{1\leq i_{1},\dots,i_{r}\leq r}\f_{j_{1}}(x_{i_{1}})\wedge\cdots\wedge\f_{j_{r}}(x_{i_{r}})\\
 & =\sum_{\s\in\fraks_{r}}\f_{j_{1}}(x_{\s^{-1}(1)})\wedge\cdots\wedge\f_{j_{r}}(x_{\s^{-1}(r)})\\
 & =\sum_{\s\in\fraks_{r}}\sgn{\s}\f_{j_{\s(1)}}(x_{1})\wedge\cdots\wedge\f_{j_{\s(r)}}(x_{r}).
\end{align*}
 Comparing the above two expression we get the desired result.
\end{proof}

\section{\label{sec:Explicit-basis}Explicit basis of the de Rham cohomology
group}

In this section, applying Proposition \ref{prop:exp-basis-1}, we
give two examples of the choice of a basis of the cohomology group
$H^{r}(\W^{\bu}(Y),\na_{\tilde{\w}})$ defined for the Veronese point
$w\in Z_{r+1}$. We adopt the notation in Sections \ref{sec:Main-theorem}
and \ref{sec:Proof-of-theorem}. Note that the parameter $\a\in\C^{N}$
in the $1$-form $\w$ (and $\tilde{\w}$) satisfies the conditions 

\begin{equation}
\a_{0}^{(1)}+\cdots+\a_{0}^{(\ell)}=0\label{eq:basis-0}
\end{equation}
and 
\[
\a_{n_{k}-1}^{(k)}\begin{cases}
\notin\Z & \text{if}\,n_{k}=1,\\
\neq0 & \text{if}\,n_{k}>1.
\end{cases}
\]

\subsection{\label{subsec:Example-1}Example 1}

We choose a basis of the cohomology group $H^{1}(\W^{\bullet}(X),\na_{\w})$
corresponding to the point $z\in Z_{2}$ as follows. Here $X=\Ps^{1}\setminus\{p^{(1)},\dots,p^{(\ell)}\}$
with the distinct points $p^{(k)}=\{[t]\in\Ps^{1}\mid tz_{0}^{(k)}=0\}$
in $\Ps^{1}$. Let $x=t_{1}/t_{0}$ be an affine coordinate of $U=\{t_{0}\neq0\}\subset\Ps^{1}$
and put $\bdx=(1,x)$. 

Let $\cW$ be the $\C$-vector space consisting of the $1$-forms
$\f$ which is expressed as a linear combination of 

\[
d(\theta_{i}(\bdx z^{(k)})),\quad1\leq k\leq\ell,0\leq i<n_{k}
\]
and is holomorphic on $X$. 

In the case where $\{p^{(1)},\dots,p^{(\ell)}\}\subset U$, we can
take
\begin{equation}
\begin{cases}
d(\theta_{i}(\bdx z^{(k)})), & 1\leq k\leq\ell,1\leq i<n_{k},\\
d(\theta_{0}(\bdx z^{(k)}))-d(\theta_{0}(\bdx z^{(\ell)})), & 1\leq k<\ell
\end{cases}\label{eq:basis-1}
\end{equation}
as a basis of $\cW$.

In the case where one member of $\{p^{(1)},\dots,p^{(\ell)}\}$, say
$p^{(\ell)}$corresponds to $x=\infty$, we can take 
\begin{equation}
\begin{cases}
d(\theta_{i}(\bdx z^{(k)})), & 1\leq k\leq\ell,1\leq i<n_{k},\\
d(\theta_{0}(\bdx z^{(k)})), & 1\leq k<\ell
\end{cases}\label{eq:basis-3}
\end{equation}
as a basis of $\cW$. We see that $\dim_{\C}\cW=N-1.$ Since any form
$\varphi\in\cW$ is $\na_{\w}$-closed, it defines a cohomology class
$[\varphi]\in H^{1}(\W^{\bullet}(X),\na_{\w}).$ 

The following proposition is known and is easily shown.
\begin{prop}
\label{prop:Basis-1}The following map is an isomorphism of vector
spaces: 
\[
\begin{array}{ccc}
\cW\slash\C\w & \lto & H^{1}(\W^{\bullet}(X),\na_{\w})\\
\varphi\ \mod.\ \C\w & \mapsto & [\varphi]
\end{array}
\]
\end{prop}

To find a basis of $\cW\slash\C\w$, recall that $\w$ can be written
as 

\begin{equation}
\w=\sum_{k=1}^{\ell-1}\a_{0}^{(k)}\{d(\theta_{0}(\bdx z^{(k)}))-d(\theta_{0}(\bdx z^{(\ell)}))\}+\sum_{1\leq k\leq\ell}\sum_{1\leq i<n_{k}}\a_{i}^{(k)}d(\theta_{i}(\bdx z^{(k)}))\label{eq:basis-4}
\end{equation}
in the case $\{p^{(1)},\dots,p^{(\ell)}\}\subset U$ by virtue of
the condition (\ref{eq:basis-0}), and 
\begin{equation}
\w=\sum_{k=1}^{\ell-1}\a_{0}^{(k)}d(\theta_{0}(\bdx z^{k}))+\sum_{1\leq k\leq\ell}\sum_{1\leq i<n_{k}}\a_{i}^{(k)}d(\theta_{i}(\bdx z^{(k)}))\label{eq:basis-5}
\end{equation}
in the case $p^{(\ell)}$corresponds to $x=\infty$ since $d(\theta_{0}(\bdx z^{(\ell)}))=0$.
Note that the $1$-form $\w$ is, in either case, a linear combination
of the basis of $\cW$ given in (\ref{eq:basis-1}) or in (\ref{eq:basis-3}).
Therefore if we omit any one member from (\ref{eq:basis-1}) (resp.
(\ref{eq:basis-3})) whose coefficient in (\ref{eq:basis-4}) (resp.
(\ref{eq:basis-5})) is nonzero, we get a basis of $\cW\slash\C\w$
and hence of the cohomology group $H^{1}(\W^{\bullet}(X),\na_{\w})$.
Let $\{\f_{1},\dots,\f_{N-2}\}$ be such basis of $H^{1}(\W^{\bullet}(X),\na_{\w})$.

Let $w=\Phi_{\l}(z)\in Z_{r+1}$ be the Veronese point corresponding
to $z\in Z_{2}$ and let $(x_{1},\dots,x_{r})$ be the coordinates
of $X^{r}.$ Let $Y=S^{r}X$ with the coordinates $y_{1},\dots,y_{r}$,
where $y_{i}$ is connected to $x_{1},\dots,x_{r}$ as the $i$-th
elementary symmetric function of $x_{1},\dots,x_{r}.$ Put $\bdy=(1,y_{1},\dots,y_{r})$
and consider the following $1$-forms on $Y$:

\begin{equation}
\begin{cases}
d(\theta_{i}(\bdy w^{(k)})), & 1\leq k\leq\ell,1\leq i<n_{k},\\
d(\theta_{0}(\bdy w^{(k)}))-d(\theta_{0}(\bdy w^{(\ell)})), & 1\leq k<\ell,
\end{cases}\label{eq:basis-2}
\end{equation}
and
\begin{equation}
\begin{cases}
d(\theta_{i}(\bdy w^{(k)})), & 1\leq k\leq\ell,1\leq i<n_{k},\\
d(\theta_{0}(\bdy w^{(k)})), & 1\leq k<\ell
\end{cases}\label{eq:basis-6}
\end{equation}
correspoinding to the case (\ref{eq:basis-1}) and (\ref{eq:basis-3}),
respectively.
\begin{lem}
\label{lem:Basis-2}We have 
\begin{equation}
d(\theta_{i}(\bdy w^{(k)}))=\sum_{1\leq j\leq r}d(\theta_{i}(\bdx_{j}z^{(k)})),\quad1\leq k\leq\ell,0\leq i<n_{k},\label{eq:basis-7}
\end{equation}
where $\bdx_{j}=(1,x_{j})$.
\end{lem}

\begin{proof}
It is enough to show (\ref{eq:basis-7}) for a fixed $k$. So the
situation can be reduced to the case where the partition $\l$ is
$(n)$. In this case $z=(z_{0},\dots,z_{n-1})\in\matt{2,n}$ and $w=\Phi_{n}(z)=(w_{0},\dots,w_{n-1})\in\matt{r+1,n}$.
We apply Proposition \ref{prop:vero-7} by setting 
\[
\bdx z=(l_{0}(x),\dots,l_{n-1}(x)),\;\bdy w=(L_{0}(y),\dots,L_{n-1}(y))
\]
and we have 
\[
\sum_{0\leq i<n}L_{i}(y)T^{i}\equiv\prod_{1\leq j\leq r}\left(\sum_{0\leq i<n}l_{i}(x_{j})T^{i}\right)\quad\mathrm{mod}\,(T^{n}).
\]
 Taking a log of both sides
\begin{align*}
\log\left(\sum_{0\leq i<n}L_{i}(y)T^{i}\right) & \equiv\sum_{1\leq j\leq r}\log\left(\sum_{0\leq i<n}l_{i}(x_{j})T^{i}\right)\\
 & \equiv\sum_{1\leq j\leq r}\left(\theta_{0}(\bdx_{j}z)+\theta_{1}(\bdx_{j}z)T+\cdots\right)\quad\mod(T^{n}).
\end{align*}
Similarly the left hand side is $\theta_{0}(\bdy w)+\theta_{1}(\bdy w)T+\cdots$.
Comparing the coefficient of $T^{i}$ of the both sides, we have 
\[
\theta_{i}(\bdy w)=\sum_{1\leq j\leq r}\theta_{i}(\bdx_{j}z)
\]
and the desired identity.
\end{proof}
Let $\psi_{1},\dots,\psi_{N-2}\in\W^{1}(Y)$ be those selected from
(\ref{eq:basis-2}) or (\ref{eq:basis-6}) in a similar manner as
$\f_{1},\dots,\f_{N-2}\in\W^{1}(X)$. Then we can applying Proposition
\ref{prop:exp-basis-1} since the condition (\ref{eq:proof-4}) is
assured by Lemma \ref{lem:Basis-2}, and we have the following result.
\begin{prop}
\label{prop:Basis-3}Let $w=\Phi_{\l}(z)$ be the Veronese point,
and let $\psi_{1},\dots,\psi_{N-2}\in\W^{1}(Y)$ be those defined
from (\ref{eq:basis-2}) or (\ref{eq:basis-6}) as explained in the
preceding paragraph. Then 
\[
\{\psi_{j_{1}}\wedge\cdots\wedge\psi_{j_{r}}\mid1\leq j_{1}<\cdots<j_{r}\le N-2\}\subset\W^{r}(Y)
\]
 gives a basis of $H^{r}(\W^{\bu}(Y),\na_{\tilde{\w}})$.
\end{prop}

\subsection{Example 2}

If we start from another basis for the cohomology group $H^{1}(\W^{\bullet}(X),\na_{\w})$,
we can obtain another set of cohomology classes which forms a basis
for $H^{r}(\W^{\bu}(Y),\na_{\tilde{\w}})$. We give an example for
it.

Let $X=\Ps^{1}\setminus\{p^{(1)},\dots,p^{(\ell)}\}$ be as in Section
\ref{subsec:Example-1} and assume that $p^{(\ell)}$ corresponds
to $x=\infty$, in which case we may take $z\in\Z_{2}$ as $z_{0}^{(\ell)}=\,^{t}(1,0)$.
\begin{prop}
\label{prop:Basis-4} Let $w=\Phi_{\l}(z)$ be the Veronese point
and assume $z_{0}^{(\ell)}=\,^{t}(1,0)$ for $z=(z^{(1)},\dots,z^{(\ell)})\in\znn$.
Then we can take a basis of $H^{1}(\W^{\bullet}(X),\na_{\w})$ the
forms

\[
\f_{i}=x^{i}dx,\quad0\leq i\leq N-3.
\]
In this case we can take a basis of $H^{r}(\W^{\bu}(Y),\na_{\tilde{\w}})$
as 
\[
S_{\mu}(y)dy_{1}\wedge\cdots\wedge dy_{r},
\]
where $\mu$ is a Young diagram contained in the shape $(\overbrace{N-r-2,\dots,N-r-2}^{r})$
and $S_{\mu}(y)$ is the polynomial in $y$ obtained from the Schur
function $s_{\mu}(x_{1},\dots,x_{r})$ for $\mu$ by representing
it as a function of elementary symmetric polynomials $y_{1},\dots,y_{r}$
of $x_{1},\dots,x_{r}.$
\end{prop}

For this proposition, see \cite{kim97-1,kim2010,tera}.

\subsection{Comments}

We make a few comments on the basis we have constructed in Propositions
\ref{prop:Basis-3}, \ref{prop:Basis-4}.

1) Taking any point $w\in Z_{r+1}$, not necessary a Veronese point,
we can define the $1$-forms $\psi_{1},\dots,\psi_{N-2}\in\W^{1}(Y)$
as in Proposition \ref{prop:Basis-3}. Then the $r$-forms 
\begin{equation}
\psi_{j_{1}}\wedge\cdots\wedge\psi_{j_{r}},\quad1\leq j_{1}<\cdots<j_{r}\le N-2,\label{eq:basis-8}
\end{equation}
which depend holomorphically on $w$, determine the classes in $H^{r}(\W^{\bu}(Y),\na_{\tilde{\w}})$.
In the case $w$ is in the Veronese image $\Phi_{\l}(Z_{2})$, these
$r$-forms give a basis of the cohomology group $H^{r}(\W^{\bu}(Y),\na_{\tilde{\w}})$,
and in particular, these classes are linearly independent. Since the
linear dependency gives a certain algebraic condition on $w\in Z_{r+1}$,
we see that these $r$-forms give linearly independent classes in
the cohomology group $H^{r}(\W^{\bu}(Y),\na_{\tilde{\w}})$ for $w$
belonging to some Zariski open subset of $Z_{r+1}$ containing the
Veronese image. 
\begin{conjecture}
For any $w\in Z_{r+1}$, the cohomolgy of the complex $(\W^{\bu}(Y),\na_{\tilde{\w}})$
is pure, $\dim_{\C}H^{r}(\W^{\bu}(Y),\na_{\tilde{\w}})=\binom{N-2}{r}$,
and the $r$-forms in (\ref{eq:basis-8}) give a basis of the cohomology
group $H^{r}(\W^{\bu}(Y),\na_{\tilde{\w}})$.
\end{conjecture}

Supporting fact for this conjecture is the holonomicity of the general
hypergeometric system (\cite{tanisaki}, Propostion 4.5), and this
system has no singular point in $Z_{r+1}$. We give an elementary
proof of the latter fact in Appendix.

2) When the partition $\l$ of $N$ is $\l=(1,\dots,1)$, the general
HGF on $Z_{r+1}$ is Gelfand's HGF. In this case, write $w\in Z_{r+1}$
as $w=(w_{1},\dots,w_{N}),\;w_{j}\in\C^{r+1}$ and put 
\[
f_{j}(y):=\bdy w_{j}=w_{0j}+w_{1j}y_{1}+\cdots+w_{rj}y_{r}.
\]
Then the $1$-forms considered in (\ref{eq:basis-2}) are written
as 
\[
\psi_{j}:=d\log f_{j}-d\log f_{N},\quad1\leq j\leq N-1
\]
and the forms (\ref{eq:basis-8}) are called the logarithmic forms.
It is known that, under an appropriate assumption on the parameter
$\a$, we can have a basis of the cohomology group $H^{r}(\W^{\bu}(Y),\na_{\tilde{\w}})$
consisting of logarithmic forms \cite{Ao-Kita-book}.

3) When the partition of $N$ is $\l=(N)$, the general HGF is called
the generalized Airy function \cite{GRS88}. For the Airy function
on $Z_{2}$, a basis of $H^{1}(\W^{\bu}(X),\na_{\w})$ is given by
\[
\psi_{i}:=d(\theta_{i}(\bdx z)),\quad1\leq i\leq N-2.
\]
This basis is an analogue of the flat basis used in computing the
intersection number for the Airy integral \cite{Irina-Ki-Na}. See
also \cite{kimura-taneda}.

\section*{Appendix}

We recall the fact about the system of differential equations satisfied
by the general HGF, and we give an elementary proof of the fact that
the system is holonomic on the generic statum $Z_{r+1}$ and has no
singularity there. 

We use $z=(z_{ij})$ for the coordinates of the matrix space $Z_{r+1}$
and let $\pa_{ij}=\pa/\pa z_{ij}$. The Lie algebra $\fj(n)$ of Jordan
group $J(n)$ is 
\[
\fj(n)=\left\{ \sum_{0\leq j<n}X_{j}\La^{j}\mid X_{j}\in\C\right\} ,
\]
where $\La=\La_{n}$ is the shift matrix of size $n$, and $\ha_{\l}=\fj(n_{1})\oplus\cdots\oplus\fj(n_{\ell})$
is the Lie algebra of $H_{\l}$. Let $\a:\ha_{\l}\to\C$ be the linear
map, where the restriction $\a|_{\fj(n_{k})}$ is given by

\[
(\a|_{\fj(n_{k})})(X_{0}^{(k)}\cdot I+X_{1}^{(k)}\La+\cdots+X_{n_{k}-1}^{(k)}\La^{n_{k}-1})=\sum_{j}\a_{j}^{(k)}X_{j}^{(k)},
\]
 $\a_{j}^{(k)}$ being the parameters in Definition \ref{def:hyp-3}. 
\begin{prop}
(\cite{KHT92}) The general hypergeometric function $F(z)$ is a solution
of the system:
\begin{equation}
\begin{cases}
\square_{ij,pq}F=0 & \quad0\leq i,j\leq r,1\leq p,q\leq N,\\
\{\tr(^{t}(zX)\pa)-\a(X)\}F=0 & \quad X\in\ha_{\l},\\
\{\tr(^{t}(Yz)\pa)+\tr(Y)\}F=0 & \quad Y\in\gl{r+1},
\end{cases}\label{eq:proof-7}
\end{equation}
 where $\square_{ij,pq}=\pa_{ip}\pa_{jq}-\pa_{iq}\pa_{jp}$, $\pa=(\pa_{ip})_{0\leq i\leq r,1\leq p\leq N}$,
and $\gl{r+1}$is the Lie algebra of $\GL{r+1}$.
\end{prop}

We call this system the general hypergeometric system of type $\l$.

Let $\clam$ be the algebraic $\cD$-module on $Z:=Z_{r+1}$ associated
with the system (\ref{eq:proof-7}). Let $\mathrm{Ch}(\clam)\subset T^{*}Z$
be the characteristic variety of $\clam$, namely the set of common
zeros of the principal symbols of all the elements of the ideal generated
by the differential operators appeared in (\ref{eq:proof-7}). Here
$T^{*}Z=Z\times\mat{r+1,N}$ is the contangent bundle of $Z$ with
the fiber coordinates $\xi=(\xi_{ij})$. We want to show the following
\begin{prop}
\label{prop:appen-1}We have $\mathrm{Ch}(\clam)=\{(z,\xi)\in T^{*}Z\mid\xi=0\}$.
Hence $\clam$ is holonomic and has no singularity on $Z$.
\end{prop}

Let $V\subset T^{*}Z$ be the zero locus of the symbols of the differential
operators in (\ref{eq:proof-7}):

\begin{align}
\s(\square_{ij,pq}) & =\xi_{ip}\xi_{jq}-\xi_{iq}\xi_{jp}=0\label{eq:app-1}\\
\s(\tr(^{t}(zX)\pa)-\a(X)) & =\tr(^{t}(zX)\xi)=0\label{eq:app-2}\\
\s(\tr(^{t}(Yz)\pa)+\tr(Y)) & =\tr(^{t}(Yz)\xi)=0\label{eq:app-3}
\end{align}
where $X$ and $Y$ are arbitrary elements of $\ha_{\l}$ and $\gl{r+1}$,
respectively. For the proof of Proposition \ref{prop:appen-1}, it
is enough to prove the following lemma.
\begin{lem*}
$V=\{(z,\xi)\in T^{*}Z\mid\xi=0\}.$
\end{lem*}
\begin{proof}
Take $(z,\xi)\in V$. From $\xi_{ip}\xi_{jq}-\xi_{iq}\xi_{jp}=0$
for any $0\leq i,j\leq r,1\leq p,q\leq N,$ we see that the rank of
$\xi$ is less than $1.$ It implies that $\xi$ can be written as 

\[
\xi=a\cdot\,^{t}b
\]
 for some column vectors $a\in\C^{r+1},b\in\C^{N}.$ According as
the partition $\l=(n_{1},\dots,n_{\ell})$, we write $b$ in a block
form

\[
b=\left(\begin{array}{c}
b^{(1)}\\
\vdots\\
b^{(\ell)}
\end{array}\right),\quad b^{(k)}=\left(\begin{array}{c}
b_{0}^{(k)}\\
\vdots\\
b_{n_{k}-1}^{(k)}
\end{array}\right).
\]
 Then we have, for any $k,$ 
\[
\xi^{(k)}=a\cdot\,^{t}b^{(k)}.
\]
 Putting these into the equations (\ref{eq:app-2}), (\ref{eq:app-3}),
we have 

\begin{align}
^{t}az^{(k)}\La_{n_{k}}^{m}b^{(k)} & =0,\quad1\leq k\leq\ell,\,0\leq m<n_{k},\label{eq:app-4}\\
zb & =0\label{eq:app-5}
\end{align}
We want to show $\xi=0.$ If $b=0$, then we are done. So we assume
$b\neq0.$ Then we can choose, for each $k,$the integer $m_{k}$
such that $0\leq m_{k}\leq n_{k}$ and 
\begin{equation}
b_{m_{k}-1}\neq0,\;b_{m_{k}}=\cdots=b_{n_{k}-1}=0.\label{eq:app-6}
\end{equation}
We have the subdiagram $\mu=(m_{1},\dots,m_{\ell})$ of $\l$. Since
$b\neq0$, we must have $|\mu|=m_{1}+\cdots+m_{\ell}>r+1.$ In fact,
if $|\mu|\leq r+1,$we can choose a subdiagram $\nu=(p_{1},\dots,p_{\ell})\subset\l$
satisfying $\mu\subset\nu$ and $|\nu|=r+1$. Then, (\ref{eq:app-5})
reads $z_{\nu}b_{\nu}=0$, where 
\[
b_{\nu}=\,^{t}(b_{0}^{(1)},\dots,b_{p_{1}-1}^{(1)},\dots,b_{0}^{(\ell)},\dots,b_{p_{\ell}-1}^{(\ell)}).
\]
But $\det z_{\nu}\neq0$ by the assumption $z\in Z$, it means that
$b_{\nu}=0$ and hence $b=0$. Thus we must have $|\mu|>r+1.$ Then
from (\ref{eq:app-4}) we have $^{t}a\cdot z_{\mu}=0.$ But from $|\mu|>r+1$
and $z\in Z,$ we have $rank\:z_{\mu}=r+1.$ It follows that $a=0$
and hence $\xi=0$. 
\end{proof}

\end{document}